\documentclass{article}

\usepackage{amsfonts, amsthm, amsmath, amssymb}
\usepackage{mathrsfs}
\usepackage[english]{babel}
\usepackage{graphicx}

\usepackage{babelbib}

\theoremstyle{plain}
  \newtheorem{theorem}{Theorem}
  \newtheorem{proposition}[theorem]{Proposition}
  \newtheorem{lemma}[theorem]{Lemma}
  \newtheorem{corollary}[theorem]{Corollary}
\theoremstyle{definition}
  \newtheorem{definition}[theorem]{Definition}
  
\theoremstyle{remark}
  
  \newtheorem{remark}[theorem]{Remark}

\newcommand{\displaycomma}{\textrm{,}}
\newcommand{\comma}{\textrm{,}}

\newcommand{\period}{\textrm{.}}
\newcommand{\bra}[1]{\left( #1 \right)}
\newcommand{\pa}[1]{\left( #1 \right)}

\newcommand{\sqa}[1]{\left[ #1 \right]}
\newcommand{\cur}[1]{\left\{ #1 \right\}}
\newcommand{\ang}[1]{\left< #1 \right>}
\newcommand{\abs}[1]{\left| #1 \right|}

\newcommand{\norm}[1]{\left\| #1 \right\|}
\newcommand{\qua}[1]{\left[ #1 \right]}

\newcommand{\rnum}{\mathbb{R}}

\newcommand{\prob}{\mathbb{P}}

\newcommand{\E}{\mathbb{E}}

\newcommand{\Prob}{\mathbb{P}}

\newcommand{\cD}{\mathbb{D}}

\begin{document}

\title{ BV-regularity for the Malliavin Derivative of the Maximum of the Wiener Process }
\author{Dario Trevisan}






\maketitle

\begin{abstract}
We show that, on the classical Wiener space, the random variable $\sup_{0\le t \le T} W_t$ admits a measure as second Malliavin derivative, whose total variation measure is  finite and singular w.r.t.\ the Wiener measure.
\end{abstract}

\section{Introduction}

Functions of bounded variation ($BV$) on abstract Wiener spaces were first introduced by M.~Fukushima and M.~Hino in \cite{Masatoshi2000} and  \cite{Fukushima2001} and subsequently studied with tools from geometric measure theory by L.~Ambrosio and his co-workers (see e.g.\ \cite{Ambrosio2010}, \cite{Ambrosio2010a} and \cite{Ambrosio2011}). In \cite{Pratelli2012}, M.\ Pratelli and the author investigate some properties of $BV$ functions in the classical Wiener space.

The aim of this article is to study in the $BV$ framework the Malliavin regularity of the maximum of a Wiener process, $M=\sup_{0\le t \le T} W_t$, which is well-known to be differentiable only once. This was initially motivated by possible applications, e.g.\ for the computation of Greeks for barrier options using Malliavin calculus, as in \cite{Gobet2003}, where it also is explicitly remarked that this lack of differentiability forces the introduction of many technicalities. As remarked by an anonymous referee, the results obtained here may also be connected to integration by parts formulas in convex sets in Wiener spaces, which were first studied by L.~Zambotti  in \cite{zambotti02} (see also \cite{hariya06}, \cite{funaki07} and \cite{otobe09} for further developments): this certainly requires deeper investigations. 

In this article, therefore, we focus on the second distributional derivative $D^2M$ and the main result, Theorem \ref{max1}, shows that it is a measure, with finite total variation (i.e.\ the Malliavin derivative $\nabla M$ is $BV$). Moreover, Theorem \ref{max2} shows that the total variation measure $\abs{D^2M}$, which is a finite Borel measure, is concentrated on the trajectories which attain the global maximum at least twice.

In the next section, we collect some easy or well-known facts about $BV$ functions on the classical Wiener space; then we proceed with the main results.

\section{Preliminary results}

In all what follows, let $T>0$ be fixed and let $\pa{\Omega = C_0 \bra{0,T},  \Prob, H^1_0 \sim L^2(0,T)}$, be the classical Wiener space, where we denote $L^2(0,T) = L^2\bra{\sqa{0,T}, \mathscr{L}}$ and $\mathscr{L}$ is the Lebesgue measure on the interval.

We completely refer to \cite{Pratelli2012}, Section 2, for a detailed introduction to $BV$ functions in the classical Wiener space setting: in this section we mainly recall and extend to the Hilbert space-valued case the definition together with a fundamental result. We also remark that, below, the square integrability condition is not optimal but provides an easier formulation. 

Let $(K,\abs{\cdot})$ be some separable Hilbert space and let $K\otimes L^2(0,T)$ denote the Hilbertian tensor product, which is equivalent to a suitable space of Hilbert-Schmidt operators. To keep notation simple, we write $\ang{\cdot,\cdot}$ for any scalar product.

In all what follows, given $h'\in L^2(0,T)$, we write $h \bra{t} = \int_0^t h'\bra{s}ds$. When $g$ is a cylindrical smooth function defined on $\Omega$, $\partial_h g$ denotes the directional derivative of $g$ along the direction $h \in C_0\bra{0,T}$, while \[ \partial^*_h g = \partial_h g - g \int_0^T h'\bra{s} dW_s \comma\] so that $-\partial^*_h$ is adjoint to $\partial_h$, in $L^2\bra{\Omega,\prob}$, i.e.\ when $X$ is also smooth, it holds $\Prob\sqa{ g \partial_h X} = - \Prob\sqa{X \partial_h^* g}$. $BV$ functions are precisely those r.v.'s such that the identity just written still holds, for any $h \in H^1_0$, when the l.h.s.\ is replaced with a suitable measure $DX$, with finite total variation measure, $\abs{DX}$.

\begin{definition}
A $K$-valued $X\in L^2\bra{\Omega,\Prob;K}$ is said to be of \emph{bounded variation} ($BV$) if there exists a $K\otimes L^2(0,T)$-valued measure $DX$, with finite total variation, such that, for every $k\in K$, $h' \in L^2(0,T)$  and every cylindrical smooth function $g$, it holds
\[ \int_\Omega g \, d\ang{DX,k\otimes h'} = -\E\qua{\ang{X,k}\partial_h^* g} \period\]
\end{definition}

\begin{remark}
As in \cite{Pratelli2012}, we use the following notation, which slightly differs from that commonly used in the framework of Malliavin calculus, e.g.\ in \cite{Nualart1995}: we write $\nabla X = \bra{\partial_t X}_{0\le t\le T}$ for the usual Malliavin derivative, while $DX$ stands for the measure-derivative.

In particular, the Malliavin space $\cD^{1,2}\bra{\Prob}$ consists of all the real valued random variables $X \in L^2\bra{\Prob}$ such that $DX = \nabla X. \Prob$, with $\abs{\nabla X} \in L^2\bra{\Prob}$.
\end{remark}

A key feature of $BV$ functions is summarized in the following closure result, whose proof proceeds along the same lines as in the last part of Theorem 4.1 in \cite{Ambrosio2010}, where the real-valued case is settled.

\begin{proposition}\label{prop-crit}
Given $X \in L^2\bra{\Omega,\Prob;K}$, if there exists some sequence $(X_n)_{n\ge1} \subseteq  BV \cap L^2\bra{\Prob;K}$ which converges to $X$ in $L^2\bra{\Prob;K}$, with
\[ \sup_{n\ge1} \abs{DX_n}\bra{\Omega} < \infty\comma\]
then $X$ is $BV$.
\end{proposition}

We discuss now the notion of measure-second derivative, which is precisely the regularity that we are going to prove for the maximum of the Wiener process. 

\begin{definition}
A real random variable $X \in \cD^{1,2}\bra{\Prob}$ is said to admit a measure $D^2X$ as \emph{second derivative} if the $L^2(0,T)$-valued random variable $\nabla X$ is $BV$.
\end{definition}

\begin{remark} If $X$ admits a second derivative $D^2X$, it follows that, for every $k',h' \in L^2(0,T)$ and every cylindrical smooth function $g$, we can integrate by parts twice:
\[ \int_\Omega g \, d\ang{D^2X,k'\otimes h'} = \E\qua{X\partial_k^*(\partial_h^* g)} \period\]
\end{remark}

After Lemma 1 in \cite{Pratelli2012}, any $L^2(0,T)$-valued measure can be identified with a measure on the product space $\Omega\times[0,T]$ and, by disintegration, if $X$ is a $BV$ function, then there exists some family of real-valued Borel measures $\bra{D_tX}_{0 \le t \le T }$ on $\Omega = C_0\bra{0,T}$ such that, for $h' \in L^2(0,T)$ and any smooth function $g$, it holds
\[ \int_\Omega g \, d\ang{DX, h'} = \int_0^T h'\bra{t} \int_\Omega g\, dD_t X dt \period\]
Such a family of measures is then unique up to Lebesgue-negligible sets in $[0,T]$. Moreover, if $X \in \cD^{1,2}\bra{\Prob}$, then $\mathscr{L}$-a.e.\ $t\in[0,T]$, the measure $D_t X$ is absolutely continuous with respect to $\Prob$ and its density is given by $\partial_t X$, so that one finds as a special case the usual identification of a Malliavin derivative $\nabla X$ with a process $\bra{\partial_t X}_{0\le t \le T}$.

When $K = L^2(0,T)$, since $K\otimes L^2(0,T) \sim L^2([0,T]^2, \mathscr{L}^2)$, a similar argument proves that if $X$ admits a second derivative as a measure $D^2X$, this can be identified with a real measure on the product space $\Omega\times[0,T]^2$. Moreover, it can be proved that there exists some family of real-valued measures $\bra{D^2_{s,t}X}_{0\le s,t \le T}$ on $\Omega$ such that for $k',h' \in L^2(0,T)$ and any smooth function $g$, it holds
\[ \int_\Omega g \, d\ang{D^2X, k'\otimes h'} = \int_{[0,T]^2} k'\bra{s} h'\bra{t}\sqa{ \int_\Omega g\, dD^2_{s,t}X }dsdt \period\]
As above, such a family of measures is then unique up to $\mathscr{L}^2$-negligible sets in $[0,T]^2$, i.e.\ if $\bra{D^2_{s,t}X'}_{0\le s,t \le T}$ is another family of measures with satisfies the conditions above, then $\mathscr{L}^2$-a.e.\ $\bra{s,t} \in [0,T]^2$, it holds $D^2_{s,t}X = D^2_{s,t}X'$. Finally, it is not difficult to argue that, $\mathscr{L}$-a.e.\ $t\in[0,T]$, the (real valued) random variable $\partial_t X$ is $BV$ and, $\mathscr{L}$-a.e.\ $s\in[0,T]$, $D_s\partial_t X = D^2_{s,t}X$.

\subsection{A concentration result for level sets}

This subsection develops the key argument to prove Theorem \ref{max2} below, only in a more general setting. Moreover, Proposition \ref{prop-conc} can be used also in the proof of Theorem \ref{max1}, although one could replace its use relying on finite-dimensional arguments only.

Given $X \in \cD^{1,2}\bra{\Prob}$, the Coarea formula (Theorem 3.7 in \cite{Ambrosio2010a}) implies that $\mathscr{L}$-a.e.\ $t \in \rnum$,  the set $\cur{X>t}$ has finite perimeter, i.e.\ $I_{\cur{X>t}}$ is a $BV$ function (see also Proposition 8 in \cite{Pratelli2012} for sufficient conditions such that, for a given $t$, $\cur{X>t}$ has finite perimeter). 
Here, we prove that it is always possible to identify some Borel representative $\tilde{X}$ for $X$, such that the perimeter measure is concentrated on the set where $\tilde{X} = t$: indeed, given a sequence of smooth cylindrical functions $(X_n) \subseteq \cD^{1,2}\bra{\Prob}$ fast convergent to $X$ in this space, i.e.\ $\norm{X_n - X_{n+1}} _{\cD^{1,2}} \le 2^{-n}$, define pointwise
\[ \tilde{X}\bra{\omega} = \limsup_{n\to \infty} X_n\bra{\omega} \comma\]
which is the so-called the quasi-continuous representative for $X$, with respect to the $(2,1)$-capacity (see e.g.\ \cite{Malliavin1997}, chapter IV). Using this notation, the following proposition holds true.

\begin{proposition}\label{prop-conc}
Given $X \in \cD^{1,2}\bra{\Prob}$, if $\cur{X>t}$ has finite perimeter, then
\[ \abs{D I_{\cur{X>t}}} \cur{ \omega \in \Omega \,\, \big| \,\, \tilde{X}\bra{\omega} \neq t} = 0 \comma\]
where $\abs{DI_{\cur{X>t}}}$ denotes the total variation measure of $DI_{\cur{X>t}}$.
\end{proposition}

\begin{proof}
We provide here an argument which relies upon the following general facts: the perimeter measure is a suitable restriction of the $1$-codimensional spherical Hausdorff measure $\rho_1$ (Theorem 1.3 in \cite{Ambrosio2011}); moreover $\rho_1$ is absolutely continuous with respect to the $(2,1)$-capacity, i.e.\ sets with null capacity are $\rho_1$-negligible (Theorem 9 in \cite{Feyel1992}).

Keeping the notation introduced above, from the the general theory of quasi-sure analysis it follows that the sequence $\cur{X_n\bra{\omega}}_{n\ge1}$ converges to $\tilde{X}\bra{\omega}$ for every $\omega \in \Omega$, possibly with the exception of a set of null capacity. Using the facts stated above, the convergence holds both $\Prob$-almost surely and $\rho_1$-a.e.\ and therefore it holds $\abs{D I_{\cur{X>t}}}$-a.e.

Let $h'\in L^2(0,T)$, let $\psi$ be a smooth cylindrical function on $\Omega$ and let $\phi$ be a smooth function, defined on $\rnum$, with compact support contained in the half line $(-\infty, t)$, so that $\phi_n = \phi \circ X_n$ is also a smooth cylindrical function for any $n \ge 1$. The following integration by parts holds:
\[ \int_\Omega \phi_n\psi\, d\ang{h', DI_{\cur{X>t}}} = -\E \sqa{I_{\cur{X>t}} \bra{\psi \, \partial_h\phi_n + \phi_n\, \partial^*_h\psi}}\period\]
As $n \to \infty$, the left hand side converges to \[\int_\Omega \bra{\phi \circ  \tilde{X}}  \psi \, d\ang{h', DI_{\cur{X>t}}}\] by Lebesgue's dominated convergence theorem, while the right hand side converges to
\[ -\E \sqa{I_{\cur{X>t}} \bra{\psi \, \partial_h\bra{\phi \circ X} + \bra{ \phi \circ X } \partial^*_h\psi}} = 0 \displaycomma \]
by the assumption on the support of $\phi$. Since $\phi$, $\psi$ and $h'$ are arbitrary, we conclude that
\[ \abs{DI_{\cur{X>t}}} \bra{\tilde{X}<t} = 0\period\]
Since $DI_{\cur{X>t}} = -DI_{\cur{X\le t}}$,  the case of $\{\tilde{X}>t \}$ follows similarly.
\end{proof}

\section{Main results}

For $0\le a\le b\le T$, define \[M_{[a,b]} = \sup\cur{W_t: a\le t \le b}\displaycomma \quad \sigma_{[a,b]} = \inf\cur{ a\le t \le b: W_t = M_{[a,b]}}\comma\]
and write $M = M_{[0,T]}$, $\sigma = \sigma_{[0,T]}$. It is well known (see \cite{Nualart1995}, pp.\ 91-98) that $M \in \cD^{1,p}\bra{\Prob}$ for any $p\ge1$, with $\nabla M = I_{[0,\sigma[}$, i.e.\
\begin{equation}\label{partial-tM}\partial_t M = I_{\cur{\sigma >t}} = I_{\cur{M_{[0,t]} < M_{[t,T]}}}\period \end{equation}
For any $t$, write $\Delta_t M =  M_{[t,T]}-M_{[0,t]}$, which is Mallavin differentiable and has an absolutely continuous law with bounded density $l_t$ (this last fact is elementary, since the joint law of $(M_{[0,t]},W_t)$ is explicitly known). 

\begin{theorem}\label{max1}
The random variable $M$ admits a second derivative as a measure $D^2M = \bra{D^2_{s,t} M}_{0\le s,t \le T}$, with finite total variation.

Moreover, $\mathscr{L}$-a.e.\ $t\in [0,T]$, for any $h'\in L^2(0,T)$ and any smooth cylindrical function $g$, it holds
\begin{equation}\label{chain-max}  \int_0^T h'\bra{s}  \sqa{\int_\Omega g\, dD^2_{s,t}M } ds = l_t\bra{x} \E \sqa{ g \sqa{ h\bra{\sigma_{[t,T]}} - h\bra{\sigma_{[0,t]}}}  \, \big| \, \Delta_t M = x }_{|x=0} \displaycomma\end{equation}
where the r.h.s.\ side is intended as its continuous version (which exists), evaluated at zero.
\end{theorem}


\begin{remark}
The random variable $M$, being the supremum of a family of linear functionals, is convex: from  \cite{Bogachev2010}, Corollary 6.5.5, we already recover the existence of $D^2_{h_1,h_2} M$ along every pair $h_1, h_2 \in L^2(0,T)$.
\end{remark}


\begin{theorem}\label{max2}

The finite measure $\abs{D^2M}$ is concentrated on the paths that attain the global maximum at least twice: in particular, it is singular with respect to $\Prob$.
\end{theorem}

\begin{remark}
It is also possible to refine the previous result and prove that $\mathscr{L}^2$-a.e.\ $s\le t$ $\abs{D^2_{s,t}M}$ is concentrated on the paths that attain the global maximum at least twice, once before $s$ and once after $t$.
\end{remark}

\begin{corollary}
The random time $\sigma$ is $BV$, with derivative $D\sigma$ concentrated on the trajectories that attain the global maximum at least twice.
\end{corollary}

\begin{proof}
It follows from $\sigma = \int_0^T \partial_s M ds = \ang{I_{[0,T[} , \nabla M}$.
\end{proof}

\subsection{Proof of Theorem \ref{max1}}

The existence of $D^2M$ follows from Proposition \ref{prop-crit}: we provide a sequence of functions $\bra{M_n}_{n\ge1}$ convergent to $M$ in $\cD^{1,2}\bra{\prob}$, such that every $\nabla M_n$ is $BV$ and $\abs{D^2 M_n}\bra{\Omega}$ is bounded, uniformly in $n\ge1$.

First, we need a lemma on Gaussian random walks. Given $n\ge1$, on the standard $n$-dimensional Gaussian space $\bra{\rnum^n, \gamma = \gamma_n}$, let $\bra{X_i}_{1\le i \le n}$ be the sequence of standard projections: $X_i \bra{x} = x_i$. Let $W_0 = 0$ and $W_k = \sum_{i=1}^k X_i$, for all $k\ge1$, so that $\bra{W_k}_{k=0}^n$ is a Gaussian random walk, starting at the origin, of length $n$. Write $A_n$ for the set
\[ A_n = \bigcap _{k=0}^n \cur{ W_k \le 0 } \period\]

An asymptotic estimate (in terms of $n$) of the probabilities $\gamma\bra{A_n}$ and $\gamma\bra{A_n\,|\, W_n = 0}$ will be needed, where we write
\[ \gamma\bra{A_n \,| \,W_n = 0} =  \frac{\abs{D_\gamma I_{\cur{W_n >0}}}\bra{A_n}}{P_\gamma\cur{W_n >0}} \comma \]
and $P_\gamma$ denotes the perimeter measure with respect to $\gamma$, which, in this finite-dimensional setting, coincides with the usual Euclidean perimeter measure times the continuous representative of the density $d\gamma/d\mathscr{L}$. Moreover, the link between Euclidean $BV$ functions and $BV$ functions with respect to $\gamma$, as a special case of abstract Wiener space (see e.g.\ \cite{Ambrosio2010a}), is provided by the following identity between measures (provided that both exist):
\begin{equation}
\label{eq-bv-eu}
D_\gamma u = \frac{e^{-\frac{\abs{x}^2}{2}}}{\bra{2\pi}^{n/2}} \, D u\comma 
\end{equation}
where $Du$ denotes the Euclidean measure-derivative.

\begin{lemma}\label{lemma-random-walk}
With the notations above, for every $n\ge1$,
\[ \gamma\bra{A_n} = \binom{2n}{n} \frac{1}{4^n}  \, \textrm{ and }\,  \gamma\bra{A_n\, |\, W_n = 0} = \frac{1}{n} \period\]
\end{lemma}

The following estimates are then obtained using Stirling approximation for the factorial and the explicit value $P_\gamma\cur{W_n >0}= \bra{2\pi}^{-1/2}$ (as proved in \cite{Ambrosio2010a}, Corollary 3.11).

\begin{corollary}\label{coro-random-walk}
For some absolute constant $C>0$, it holds for every $n\ge1$,
\[ \gamma\bra{A_n} \le C n^{-1/2}  \, \textrm{ and }\,  \abs{D_\gamma I_{\cur{W_n >0}}}\bra{A_n} \le C n^{-1} \period\]
\end{corollary}

\begin{proof}(of Lemma \ref{lemma-random-walk}). The first formula follows from a classical result, due to E.~Sparre Andersen (see articles \cite{andersen53}, \cite{spitzer56} and also \cite{spitzer2001principles}, page 218):
\[ \sum_{n = 0}^\infty \gamma\bra{A_n} t^n = \exp\cur{ \sum_{k=1}^\infty \frac{t^k}{k} \gamma\bra{W_k \le 0 } } \displaycomma\]
for every $0 \le t < 1$. Since $\gamma\bra{W_k \le 0 } = 1/2$,
\[ \sum_{n = 0}^\infty \gamma\bra{A_n} t^n = \bra{1-t}^{-1/2} = \sum_{n=0}^\infty \binom{2n}{n} \frac{1}{4^n} t^n\period\]

The second identity follows from a variation of an proof developed in \cite{andersen53scand}. To keep notation simple, let us introduce the probability measure 
\[ \mu = \frac{\abs{D_\gamma I_{\cur{W_n >0}}}}{P_\gamma \cur{W_n >0}}\displaycomma \]
and write
\[ B_m  = \bigcap_{i=0}^{n-1} \cur{W_i \le W_m } \cap \cur{W_n =0 } \displaycomma\]
for $m = 0, \ldots, n$, and $B_n = B_0$. Using the Euclidean theory of $BV$ functions, or adapting the proof of Proposition \ref{prop-conc}, it is not difficult to show that the measure $\mu$ is concentrated on the hyperplane $\cur{W_n = 0}$, so that $\mu\bra{B_0} = \mu\bra{A_n}$ .

We are going to show that $\mu\bra{B_m} = 1/n$ for $m=0,\ldots,n-1$. Let $\theta: \rnum^n \to \rnum^n$ be the cyclical permutation of coordinates
\[ \bra{x_1, \ldots, x_n } = x \mapsto  \theta\bra{x} = \bra{x_2, x_3, \ldots, x_n, x_1 } \displaycomma\]
and notice that both $\gamma$ and $\mu$ are $\theta$-invariant, so that
\[ \mu \bra{ \theta^{-1}\bra{A}} = \mu \bra{A} \]
for any Borel set  $A \subseteq \rnum^n$. Since $W_k\circ \theta = W_{k+1} - X_1$, for $0 \le k < n$, while $W_n\circ \theta = W_n$, it holds for $m = 0,\ldots,n-1$,
\[\theta^{-1}\bra{B_m} = \bigcap_{i=1}^{n-1} \cur{W_i \le W_{m+1} } \cap  \cur{ 0 = W_n \le W_{m+1} } = B_{m+1} \period\]
But $B_{n} = B_0$, so that $\mu\bra{B_0} = \mu\bra{B_m}$ for $m=0,\ldots,n-1$. Moreover,
\[1 =\mu\bra{\cur{W_n=0}}= \mu\bra{\bigcup_{m=0}^{n-1} B_m} = \sum_{m=0}^{n-1} \mu\bra{B_m}\displaycomma\]
where the last identity follows from the fact that, for $0\le m < k \le n-1$,
\[ B_m \cap B_{k} \subseteq \cur{ W_m - W_k = 0, W_n = 0} \displaycomma\]
and the right hand side above is neglected by $\mu$, since it is a linear subspace of codimension $2$ and it can be shown that $\mu$ is absolutely continuous with respect to the $(n-1)$-dimensional Hausdorff measure on the hyperplane $\cur{W_n = 0}$. This last fact follows either from the Euclidean theory of $BV$ functions or after Theorem 1.3 in \cite{Ambrosio2011}.
\end{proof}

We conclude with the proof of Theorem \ref{max1}. For fixed $n\ge1$, write $W_k = \sqrt{n/T} W_{kT/n}$, for $0\le k\le n$, and
\begin{equation}\label{defin} M_n = \sqrt{T/n}\max\cur{W_k:  k=0,\ldots,n} = \sqrt{T/n}\sum_{k=0}^n I_{E_k}W_k \displaycomma\end{equation}
where $E_k = \bigcap_{i=0}^n\cur{W_i\le W_k}$. 
By simple Malliavin calculus, it holds $M_n \in \cD^{1,2}\bra{\Prob}$  with
\[ \nabla M_n  = \sqrt{T/n} \sum_{k=0}^n I_{E_k}I_{[0,kT/n]} \displaycomma\]
which is a finite sum of $BV$ maps, since every $E_k$ is a finite intersection of sets with finite perimeter. Indeed, using this formula it is also easy to show that $\bra{M_n}_{n\ge1}$ converges to $M \in \cD^{1,2}\bra{\prob}$: this is the classical proof that $M$ is Malliavin differentiable. 

It is enough to prove that the sequence of total variations $\abs{D^2M_n}\bra{\Omega}$ is bounded. With a slight variation of Proposition 3.5 in \cite{Ambrosio2010}, the computation can be performed on $\bra{\rnum^n, \gamma}$, where the notation for $\bra{W_k}_{k=0}^n$ can be consistently identified with that of the lemma above. Therefore, $M_n$ is defined by means of \eqref{defin}, but it is a function defined on $\rnum^n$.

Moreover, let $(e_k)_{k=0,\ldots,n}$ be the standard basis in $\rnum^n$ and, for $k=0,\ldots, n$, let $s_k$ be the vector $\sum_{i=0}^k e_i$. The second derivative of  $M_n$ is then given by
\[ D^2M_n = \sqrt{T/n}\sum_{k=0}^n s_k \otimes DI_{E_k} = \sqrt{T/n}\sum_{k, m=0}^n I_{E_k} s_k \otimes DI_{\cur{W_m \le W_k}} \comma\]
where a Leibniz rule has been applied to show that 
\[ DI_{E_k} =\sum_{m=0}^n I_{E_k}  DI_{\cur{W_m \le W_k}} \comma \]
which can be proved by a direct computation in the Euclidean setting.

Since $ DI_{\cur{W_m \le W_k}} = - DI_{\cur{W_k \le W_m}} $, it holds
\[ D^2M_n= \sqrt{T/n}\sum_{0\le m<k\le n} I_{E_k} (s_k-s_m) \otimes DI_{\cur{W_{m} \le W_{k}}}\comma\]
so that it is sufficient for fixed $m<k$, to provide a bound to the quantity
\[ \sqrt{T/n} \abs{I_{E_k}.( s_k-s_m)  \otimes DI_{\cur{W_m \le W_k}}}\bra{\rnum^n} = \sqrt{\bra{m-k}T/n} \abs{DI_{\cur{W_m \le W_k}}} \bra{E_k}\period\]
By Proposition \ref{prop-conc}, the measure above is concentrated on $\cur{W_m = W_k}$ and therefore there is no loss of generality if we substitute $E_k$ with
\[ \bigcap_{i=0}^m \cur{W_{i} -W_{m} \le 0 } \bigcap_{i=m}^{k} \cur{W_{i} -W_{m} \le 0} \bigcap_{i= k}^n \cur{W_{i} -W_{k} \le 0 } \period\]
Moreover, due to the independence of the increments and the explicit characterization of the measure-derivative provided by \ref{eq-bv-eu}, it is not difficult to conclude that the quantity above splits into the product of three terms, where two of them are
\[ \gamma\bra{\bigcap_{i=0}^m \cur{W_{i} -W_{m} \le 0 }}  \gamma\bra{\bigcap_{i= k}^n \cur{W_{i} -W_{k} \le 0}} \comma\]
and the third is
\[ \abs{DI_{\cur{W_{m} \le W_{k}}}}\bra{ \bigcap_{i=m}^{k} \cur{W_{i} -W_{m} \le 0}}\period \]
Using the estimates established in Corollary \ref{coro-random-walk}, summing upon $0\le m < k \le n$,  it holds
\[ \abs{D^2M_n}\bra{\Omega} \le C^3 \sqrt{T} \sum_{ 0< m < k < n} \frac{1}{\sqrt{m\bra{k-m} \bra{n-k}n}} + R\bra{n} \]
where $R\bra{n}$ takes into account the contribution of the terms with $m=0$ or $k=n$. As $n\to \infty$, $R\bra{n}$ is easily seen to be infinitesimal while the sum converges to
\[\int_0^1dt\int_t^1 \frac{ds}{\sqrt{ t\bra{s-t}\bra{1-s} }} < \infty \period\]

This settles the existence of $D^2M$. Formula \eqref{chain-max} is indeed an application of Proposition 8 in \cite{Pratelli2012}, together with the fact that $\mathscr{L}$-a.e.\ $t\in[0,T]$, $D^2_{s,t}M = D_s \partial_t M$. As already remarked $\partial_t M$ is the indicator function of the zero level set for the function $\Delta_t M$, which admits $I_{[\sigma_{[0,t]}, \sigma_{[t,T]}]}$ as Malliavin derivative.

\subsection{Proof of Theorem \ref{max2}}

After Theorem \ref{max1}, $\mathscr{L}$-a.e.\ $t \in [0,T]$, $\partial_tM$ is $BV$: it is then sufficient to prove that, for every $h' \in L^2(0,T)$, the real measure $\mu = \ang{h', D \partial_tM}$ (actually, its total variation measure) is concentrated on the set of paths that attain the global maximum twice. 

As already noticed,  $\partial_tM$ is the indicator function of the set $\cur{ \Delta_t M >0}$, which is $BV$: after Proposition \ref{prop-conc},  $\mu$ is concentrated on the set $\cur{ \Delta_t M =0}$, where $\Delta_t M$ is intended as its natural representative, which is defined everywhere. When $\Delta_t M =0$, i.e.\ $M_{[0,t]} = M_{[t,T]}$, the global maximum is attained twice, with the possible exception of the case $M_{[0,t]} = W_t = M_{[t,T]}$.

Therefore, the theorem will follow if we prove that $\mu$ is always concentrated on the set $A \cup B$, where $A= \cur{M_{[0,t]} > W_t}$ and $B = \cur{M_{[t,T]} > W_t}$. In particular, we are going to prove that, if $h' = 0$ a.e.\ on $[0,t]$, then $\mu$ is concentrated on  $A$ while, if $h' = 0$ a.e.\ on $[t,T]$, it is concentrated on $B$: then, the general case follows simply decomposing $h' = h'I_{[0,t[}+ h'I_{[t,T]}$.

The argument is a slight variation of the proof of Proposition \ref{prop-conc}, and the two cases are treated similarly: for brevity we consider only the case $h' = 0$ a.e.\ on $[0, t]$.

Let $\psi$ be some smooth cylindrical function and let $\phi$ be a smooth function, defined on $\rnum$, with compact support contained in the half line $(-\infty, \epsilon)$, for some $\epsilon >0$, with $0 \le \phi \le 1$. Arguing as in the proof of Proposition \ref{prop-conc}, i.e.\ by fast approximating $M_{[0,t]} - W_t$ in $\cD^{1,2}\bra{\Prob}$ and passing to the limit, the following integration by parts holds:
\[ \int_\Omega \phi\circ\bra{M_{[0,t]} - W_t} \psi d \mu = -\E\sqa{ I_{\cur{\Delta_t M >0}} \bra{\phi\circ\bra{M_{[0,t]} - W_t} } \partial^*_h \psi} \comma\]
since $\partial_{h}\phi\bra{M_{[0,t]} - W_t} = 0$ by the assumption on the support of $h'$ and the fact that $\partial_s\bra{M_{[0,t]} - W_t} = 0$ for $s \in [t,T]$. Being $\phi$ arbitrary, it follows by H\"older inequality that, for some constant $C$ depending  only $\psi$ and $h'$, it holds
\[ \abs{ \bra{\psi \mu } \cur{M_{[0,t]} - W_t<\epsilon} } \le C\bra{\psi, h'} \Prob\bra{M_{[0,t]} - W_t \le \epsilon }^{1/2}\comma\]
so that, as $\epsilon$ goes to zero, we conclude by dominated convergence that
\[ \bra{\psi \mu}\bra{\Omega \setminus A} =  \bra{\psi \mu} \cur{M_{[0,t]} - W_t\le 0}  = 0 \comma \]
which leads to the thesis, being $\psi$ also arbitrary.

\section*{Acknowledgments}
The author thanks L.~Ambrosio and M.~Pratelli for their support during the development of this article.

\bibliography{max-prob}

\end{document}